\documentclass[english]{scrartcl}
\usepackage[T1]{fontenc}
\usepackage[latin9]{inputenc}
\usepackage[a4paper]{geometry}
\usepackage{babel}
\usepackage{refstyle}
\usepackage{amsthm}
\usepackage{amsmath}
\usepackage{amssymb}
\usepackage{esint}
\usepackage[authoryear]{natbib}
\usepackage{lmodern}
\usepackage{bbm}
\usepackage[cmyk]{xcolor}
\usepackage[unicode=true,pdfusetitle,
 bookmarks=true,bookmarksnumbered=false,bookmarksopen=false,
 breaklinks=false,pdfborder={0 0 0},backref=false,colorlinks=true]
 {hyperref}

\makeatletter

\theoremstyle{plain}
\newtheorem{thm}{Theorem}
\newtheorem{prop}[thm]{Proposition}
\newtheorem{lem}[thm]{Lemma}
\newtheorem{cor}[thm]{Corollary}
\newtheorem{hyp}[thm]{Hypothesis}
\theoremstyle{definition}
\newtheorem{defn}[thm]{Definition}
\theoremstyle{remark}
\newtheorem{rem}[thm]{Remark}

\global\long\def\e{\mathrm{e}}
\global\long\def\rd{\mathrm{d}}
\global\long\def\bu{\boldsymbol{u}}
\global\long\def\bv{\boldsymbol{v}}

\global\long\def\bx{\boldsymbol{x}}

\global\long\def\be{\boldsymbol{e}}

\global\long\def\bphi{\boldsymbol{\varphi}}
\global\long\def\bpsi{\boldsymbol{\psi}}

\global\long\def\bnabla{\boldsymbol{\nabla}}
\global\long\def\bcdot{\boldsymbol{\cdot}}
\global\long\def\bwedge{\boldsymbol{\wedge}}
\global\long\def\bzero{\boldsymbol{0}}
\global\long\def\widebar#1{\bar{#1}}
\global\long\def\BMO{B\mkern-1.5muM\mkern-2muO}

\begin{document}

\title{On the asymptotic stability of\\
steady flows with nonzero flux\\
in two-dimensional exterior domains}

\author{\href{mailto:guill093@umn.edu}{Julien Guillod}\\
{\small{School of Mathematics}}\\
{\small{University of Minnesota}}}

\date{May 12, 2016}
\maketitle
\begin{abstract}
The Navier-Stokes equations in a two-dimensional exterior domain are
considered. The asymptotic stability of stationary solutions satisfying
a general hypothesis is proven under any $L^{2}$-perturbation. In
particular the general hypothesis is valid if the steady solution
is the sum of the critically decaying flux carrier with flux $\left|\Phi\right|<2\pi$
and a small subcritically decaying term. Under the central symmetry
assumption, the general hypothesis is also proven for any critically
decaying steady solutions under a suitable smallness condition.
\end{abstract}

\paragraph{MSC class}
35Q30, 35B35, 76D05

\section{Introduction}

We consider the Navier-Stokes equations in a domain $\Omega=\mathbb{R}^{2}\setminus\widebar B$
where $B$ is a bounded simply connected Lipschitz domain,
\begin{equation}
\begin{aligned}\partial_{t}\bu+\bu\bcdot\bnabla\bu & =\Delta\bu-\bnabla p\,,\qquad\qquad & \bnabla\bcdot\bu & =0\,,\\
\bu\big|_{t=0} & =\bu_{0}\,, & \bu\big|_{\partial\Omega} & =\bu^{*}\,,
\end{aligned}
\label{eq:ns-u}
\end{equation}
with an inhomogeneous boundary condition $\bu^{*}$. Without loss
of generality, we will suppose that $\bzero\in B$. Therefore, there
exists $R>\varepsilon>0$, such that $B(\bzero,\varepsilon)\cap\Omega=\emptyset$
and $\widebar B\subset B(\bzero,R)$, where $B(\bzero,r)$ denotes
the open ball of radius $r$ centered at the origin. Under smallness
assumptions on $\bu^{*}$, we expect that the long-time behavior of
the solution to this system will be close to the corresponding steady-state,
\begin{align}
\Delta\bar{\bu}-\bnabla\bar{p} & =\bar{\bu}\bcdot\bnabla\bar{\bu}\,, & \bnabla\bcdot\bar{\bu} & =0\,, & \bar{\bu}\big|_{\partial\Omega} & =\bu^{*}\,.\label{eq:ns-ubar}
\end{align}
We note that for large values of the forcing $\bu^{*}$, the long-time
behavior is in general not expected to be close to the steady state
due to the presence of turbulence. In all what follows, we will only
consider the case where $\bar{\bu}$ decays to zero at infinity,
\[
\lim_{|\bx|\to\infty}\bar{\bu}=\bzero\,.
\]
The existence of such solutions for this steady-state problem is still
open in general \citet[Chapter I]{Galdi-IntroductiontoMathematical2011}.
Some hypotheses on the boundary condition $\bu^{*}$ ensuring the
existence of such a steady solution are explained later on.

For vanishing forcing $\bu^{*}=\bzero$ or in $\mathbb{R}^{2}$, \citet{Borchers-L2decayNavier1992}
established the asymptotic stability of trivial solution $\bar{\bu}=\bzero$
under $L^{2}$-perturbations. Even under more general hypotheses,
the solution is known to be asymptotic to the Oseen vortex \citep{Gallay.Wayne-Invariantmanifoldsand2002,Gallay.Wayne-Globalstabilityof2005,Iftimie-Selfsimilarasymptotics2011,Gallay-Longtimeasymptotics2013,Ervedoza-LongTimeDisk2014,Maekawa-AsymptoticStabilityWeak2015}.
The aim of this paper is to stud the stability of the steady solution
for a nonzero forcing $\bu^{*}$. More precisely, we introduce a general
hypothesis on the steady solutions such that their asymptotic stability
with respect to any $L^{2}$-perturbation can be proven by the energy
method. We will also determine some explicit criteria on the steady
solution implying the validity of the main hypothesis. The proof will
follow the ideas developed by \citet{Karch.Pilarczyk-AsymptoticStabilityof2011}
for stability of the Landau solutions for the three-dimensional case,
the main difference being \lemref{limit-zero}. In particular for
$\bar{\bu}=\bu^{*}=\bzero$, this later lemma allow us to propose
a simpler proof of the results of \citet{Masuda-Weaksolutionsof1984,Borchers-L2decayNavier1992}.

We now introduce the concept of criticality for the steady solutions
\citep[\S 1]{Guillod-review2015} in order to explain the cases which
can be treated by our method. The Navier-Stokes equations \eqref{ns-ubar}
in $\mathbb{R}^{2}$ have the following scaling symmetry, $\bar{\bu}\mapsto\lambda\bar{\bu}(\lambda\bx)$.
A scale-invariant steady solution $\bar{\bu}$ therefore has to decay
like $\left|\bx\right|^{-1}$ at infinity, and we call this decay
critical. The steady solutions decaying strictly faster than the critical
decay, more precisely bounded by $\left(\left|\bx\right|\log\left|\bx\right|\right)^{-1}$
is called subcritically decaying. To our knowledge, all the known
solutions of \eqref{ns-ubar} are of the form of a critically decaying
term plus a subcritically decaying term \citep{Jeffery-two-dimensionalsteadymotion1915,Hamel-SpiralfoermigeBewegungen1917,Galdi-StationaryNavier-Stokesproblem2004,Yamazaki-stationaryNavier-Stokesequation2009,Yamazaki-Uniqueexistence2011,Pileckas-existencevanishing2012,Hillairet-mu2013,Guillod-Generalizedscaleinvariant2015}.
The only notable exception is the solutions constructed by \citet[\S11]{Hamel-SpiralfoermigeBewegungen1917},
\begin{align}
\bar{\bu}_{\Phi,\mu,A} & =\frac{\Phi}{2\pi r}\be_{r}+\left(\frac{\mu}{2\pi r}+A\gamma(r)\right)\be_{\theta}\,, & \gamma(r) & =\begin{cases}
r^{1+\frac{\Phi}{2\pi}}\,, & \Phi\neq-4\pi\,,\\
\frac{\log r}{r}\,, & \Phi=-4\pi\,.
\end{cases}\label{eq:hamel}
\end{align}
where $A,\Phi,\mu\in\mathbb{R}$. For $A\neq0$ and $-4\pi\leq\Phi<-2\pi$,
the Hamel solution \eqref{hamel} is supercritically decaying, otherwise
it is critically decaying.

In contrast to the three-dimensional case studied by \citet{Heywood-stationaryStability1970,Borchers.Miyakawa-stabilityofexterior1995},
in two dimensions the Hardy inequality has a logarithmic correction
and therefore cannot be used to show that any critically decaying
steady solution satisfies our general hypothesis. We will show that
the critically decaying flux carrier $\bar{\bu}_{\Phi,0,0}$ with
$\left|\Phi\right|<2\pi$ satisfies the general hypothesis as well
as any subcritical solutions. We remark that our stability result
breaks down with respect to the size of the flux exactly at the same
value ($\Phi=-2\pi$) where the uniqueness of the steady solution
of \eqref{ns-ubar} is invalidated by the Hamel solution \eqref{hamel},
see for example \citet[\S XII.2]{Galdi-IntroductiontoMathematical2011}.
Moreover, we show that our general hypothesis is satisfied if everything
is assumed to be centrally symmetric. This improves the results of
\citet{Galdi-Stability2D2015,Yamazaki-RateConvergence2016} which
require axial symmetries with respect to both coordinates axes. However,
we will show that the pure rotating solution $\bar{\bu}_{0,\mu,0}$
with $\mu\neq0$ does not satisfy our general hypothesis without assuming
the central symmetry. Very recently, \citet{Maekaw-stabilitysteadycircular2016}
announced the asymptotic stability of the pure rotating solution in
the exterior of a disk under the assumption that both $\mu$ and the
$L^{2}$-perturbations are small, by calculating the spectrum of the
linearized operator almost explicitly. To our knowledge, this is the
only result together with the present ones showing the stability of
a nontrivial steady state for the two-dimensional Navier-Stokes equations.

\paragraph{Notations}

The space of smooth solenoidal functions having compact support in
$\Omega$ is denoted by $C_{0,\sigma}^{\infty}(\Omega)$. The completion
of $C_{0,\sigma}^{\infty}(\Omega)$ in the natural norm associated
to $L^{2}(\Omega)$, $H^{1}(\Omega)$, and $\dot{H}^{1}(\Omega)$
are denoted by $L_{\sigma}^{2}(\Omega)$, $H_{0,\sigma}^{1}(\Omega)$,
and $\dot{H}_{0,\sigma}^{1}(\Omega)$ respectively. The Sobolev space
$H_{\sigma}^{1}(\Omega)$ and its homogeneous counterpart $\dot{H}_{\sigma}^{1}(\Omega)$
denote the space of weakly divergence-free vector fields respectively
in $H^{1}(\Omega)$ and in $\dot{H}^{1}(\Omega)$. See for example
\citet[Chapters II \& III]{Galdi-IntroductiontoMathematical2011}
for the standard properties of these spaces.

\section{Main results}

We make the following general hypothesis which is required to show
the stability by the energy method:

\begin{hyp}\label{hyp:main}

Assume there exists $\delta\in[0,1)$ and a solution $\bar{\bu}\in\dot{H}_{\sigma}^{1}(\Omega)$
of \eqref{ns-ubar} (for some $\bu^{*}$) such that for all $\bv\in\dot{H}_{0,\sigma}^{1}(\mathbb{R}^{2})$,
\[
\big(\bv\bcdot\bnabla\bv,\bar{\bu}\big)\leq\delta\left\Vert \bnabla\bv\right\Vert _{2}^{2}\,.
\]
\end{hyp}
\begin{rem}
The linearized operator of \eqref{ns-ubar} around $\bar{\bu}$ is
defined on $H_{0,\sigma}^{1}(\Omega)$ by
\[
\mathcal{L}\bv=-\Delta\bv+\mathbb{P}\left(\bar{\bu}\bcdot\bnabla\bv\right)+\mathbb{P}\left(\bv\bcdot\bnabla\bar{\bu}\right)\,,
\]
where $\mathbb{P}$ is the Leray projection. This operator satisfies
\[
\bigl(\mathcal{L}\bv,\bv\bigr)=\left\Vert \bnabla\bv\right\Vert _{2}^{2}-\bigl(\bv\bcdot\bnabla\bv,\bar{\bu}\bigr)\,,
\]
and therefore is positive definite if and only if \hypref{main} holds.
\end{rem}
By defining
\begin{align*}
\bu & =\bar{\bu}+\bv\,, & \bu_{0} & =\bar{\bu}+\bv_{0}\,,
\end{align*}
the original system \eqref{ns-u} becomes 
\begin{equation}
\begin{aligned}\partial_{t}\bv+\bar{\bu}\bcdot\bnabla\bv+\bv\bcdot\bnabla\bar{\bu}+\bv\bcdot\bnabla\bv & =\Delta\bv-\bnabla q\,, & \bnabla\bcdot\bv & =0\,,\\
\bv\big|_{t=0} & =\bv_{0}\,, & \bv\big|_{\partial\Omega} & =\bzero\,.
\end{aligned}
\label{eq:ns-v}
\end{equation}

By using \hypref{main}, the existence of weak solutions can be shown
by the standard method:
\begin{defn}
\label{def:weak-solution}A weak solution $\bv$ of \eqref{ns-v}
is a vector-field $\bv\in L^{\infty}(0,\infty;L_{\sigma}^{2}(\Omega))\cap L^{2}(0,\infty;\dot{H}_{0,\sigma}^{1}(\Omega))$
such that
\begin{equation}
\begin{aligned}\bigl(\bv(t),\bphi(t)\bigr)+\int_{s}^{t}\Big[\bigl(\bnabla\bv,\bnabla\bphi\bigr)+\bigl(\bar{\bu}\bcdot\bnabla\bv,\bphi\bigr)+\bigl(\bv\bcdot\bnabla\bar{\bu},\bphi\bigr)+\bigl(\bv\bcdot\bnabla\bv,\bphi\bigr)\Big]\rd\tau\qquad\qquad\qquad\\
=\big(\bv(s),\bphi(s)\bigr)+\int_{s}^{t}\bigl(\bv,\dot{\bphi}\bigr)\rd\tau\,,
\end{aligned}
\label{eq:def-weak-solution}
\end{equation}
for all $s\geq t\geq0$ and $\bphi\in C^{1}(0,\infty,C_{0,\sigma}^{\infty}(\Omega))$.\end{defn}
\begin{thm}
\label{thm:existence}Assuming \hypref{main} holds, for $\bv_{0}\in L_{\sigma}^{2}(\Omega)$
and $T>0$, there exists a weak solution $\bv$ of \eqref{ns-v} in
the energy space
\[
X_{T}=L^{\infty}(0,T;L_{\sigma}^{2}(\Omega))\cap L^{2}(0,T;\dot{H}_{0,\sigma}^{1}(\Omega))\,,
\]
satisfying the energy inequality
\begin{equation}
\left\Vert \bv(t)\right\Vert _{2}^{2}+\left(1-\delta\right)\int_{s}^{t}\left\Vert \bnabla\bv(\tau)\right\Vert _{2}^{2}\rd\tau\leq\left\Vert \bv(s)\right\Vert _{2}^{2}\,.\label{eq:energy-inequality}
\end{equation}

\end{thm}
The main result of the paper is the asymptotic $L^{2}$-stability
of the steady solutions satisfying \hypref{main}.
\begin{thm}
\label{thm:stability}Assuming \hypref{main} holds, for $\bv_{0}\in L_{\sigma}^{2}(\Omega)$
and $\bv$ a weak solution of \eqref{ns-v} on $(0,\infty)$ satisfying
\eqref{energy-inequality}, then
\[
\lim_{t\to\infty}\left\Vert \bv(t)\right\Vert _{2}=0\,.
\]

\end{thm}

\section{Conditions for the validity of the main hypothesis}

In this section, we will deduce some explicit conditions on $\bar{\bu}$
which ensure the validity of our main \hypref{main} and therefore
its stability. We will prove that the critically decaying flux carrier
$\bar{\bu}_{\Phi,0,0}$ and subcritically decaying steady solutions
satisfy \hypref{main} under suitable smallness assumptions. We will
prove that the steady solution $\bar{\bu}_{0,\mu,0}$ corresponding
to a rotating cylinder doesn't satisfies \hypref{main}. Moreover
under the central symmetry assumption, any small critically decaying
solution satisfies \hypref{main}.

The most simple case concerning subcritically decaying steady solutions
is obtained by applying the Hardy inequality:
\begin{lem}
\textup{\label{lem:hardy}If for some $\delta>0$,} 
\begin{equation}
\sup_{\bx\in\Omega}\left(\left|\bx\right|\log(\varepsilon^{-1}\left|\bx\right|)\left|\bar{\bu}(\bx)\right|\right)\leq\frac{\delta}{2}\,,\label{eq:subcritical}
\end{equation}
then for all $\bv\in\dot{H}_{0,\sigma}^{1}(\Omega)$,
\[
\big|\big(\bv\bcdot\bnabla\bv,\bar{\bu}\big)\big|\le\delta\left\Vert \bnabla\bv\right\Vert _{2}^{2}\,.
\]
\end{lem}
\begin{proof}
This is a simple consequence of the Hardy inequality \citep[Theorem II.6.1]{Galdi-IntroductiontoMathematical2011},
\[
\left\Vert \frac{\bv}{\left|\bx\right|\log(\varepsilon^{-1}\left|\bx\right|)}\right\Vert \leq2\left\Vert \bnabla\bv\right\Vert _{2}\,.
\]

\end{proof}
By assuming that the field $\bv$ is centrally symmetric,
\[
\bv(\bx)=-\bv(-\bx)\,,
\]
we can obtain an Hardy inequality without logarithmic correction,
and therefore can treat the stability of critically decaying steady
solutions:
\begin{lem}
\label{lem:hardy-central}There exists some constant $C>0$ depending
only on the domain $\Omega$, such that if
\begin{equation}
\sup_{\bx\in\Omega}\left(\left|\bx\right|\left|\bar{\bu}(\bx)\right|\right)\leq\frac{\delta}{C}\,,\label{eq:critical}
\end{equation}
for some $\delta>0$, then for any centrally symmetric $\bv\in\dot{H}_{0,\sigma}^{1}(\Omega)$,
\[
\big|\big(\bv\bcdot\bnabla\bv,\bar{\bu}\big)\big|\le\delta\left\Vert \bnabla\bv\right\Vert _{2}^{2}\,.
\]
\end{lem}
\begin{proof}
It suffices to prove the following Hardy inequality
\begin{equation}
\left\Vert \frac{\bv}{\left|\bx\right|}\right\Vert _{2}\leq C\left\Vert \bnabla\bv\right\Vert _{2}\,,\label{eq:hardy-central}
\end{equation}
for all centrally symmetric $\bv\in\dot{H}_{0,\sigma}^{1}(\Omega)$.
Let $R>0$ be such that $B\subset B(\bzero,R)$. We denote by $B_{n}$
the ball $B_{n}=B(\bzero,nR)$ and by $S_{n}$ the shell
\begin{align*}
S_{0} & =B_{1}\setminus B\,, & S_{n} & =B_{2n}\setminus B_{n}\,,\quad\text{for}\quad n\geq1\,.
\end{align*}
Since $\bv$ is zero on $\partial B$, by using the PoincarÃ© inequality
in $S_{0}$, there exists a constant $C_{0}>0$ such that
\[
\bigl\Vert\bu;L^{2}(S_{0})\bigr\Vert^{2}\leq C_{0}\bigl\Vert\bnabla\bu;L^{2}(S_{0})\bigr\Vert^{2}\,.
\]
Since $\bv$ is centrally symmetric, we have
\[
\int_{\gamma}\bv=0\,,
\]
for $\gamma$ any centrally symmetric smooth curve. Therefore for
$n\geq1$,
\[
\int_{S_{n}}\bv=0\,,
\]
and by using the PoincarÃ© inequality in $S_{n}$, there exists $C_{n}>0$
such that
\[
\bigl\Vert\bv;L^{2}(S_{n})\bigr\Vert^{2}\leq C_{n}\bigl\Vert\bnabla\bv;L^{2}(S_{n})\bigr\Vert^{2}\,.
\]
By hypothesis $\left|\bx\right|\geq\varepsilon$ so we obtain
\[
\bigl\Vert\bv/\left|\bx\right|;L^{2}(S_{n})\bigr\Vert^{2}\leq\frac{C_{n}}{\varepsilon}\bigl\Vert\bnabla\bv;L^{2}(S_{n})\bigr\Vert^{2}\,.
\]
But the domains $S_{n}$ are scaled versions of $S_{1}$, \emph{i.e.}
$S_{n}=nS_{1}$ for $n\geq1$ and therefore, since the two norms in
the previous inequality are scale invariant, we obtain that $C_{n}=C_{1}$,
for $n\geq1$. Now we have for $N\geq1$,
\begin{align*}
\bigl\Vert\bv/\left|\bx\right|;L^{2}(B_{2N}\setminus B)\bigr\Vert^{2} & =\sum_{n=0}^{N}\bigl\Vert\bv/\left|\bx\right|;L^{2}(S_{n})\bigr\Vert^{2}\leq\frac{1}{\varepsilon}\sum_{n=0}^{N}C_{n}\bigl\Vert\bnabla\bu;L^{2}(S_{n})\bigr\Vert^{2}\\
 & \leq\frac{C_{0}+C_{1}}{\varepsilon}\sum_{n=0}^{N}\bigl\Vert\bnabla\bu;L^{2}(S_{n})\bigr\Vert^{2}\leq\frac{C_{0}+C_{1}}{\varepsilon}\bigl\Vert\bnabla\bu;L^{2}(B_{2N}\setminus B)\bigr\Vert^{2}\,.
\end{align*}
Finally, by taking the limit $N\to\infty$, \eqref{hardy-central}
is proven where $C^{2}=\varepsilon^{-1}(C_{0}+C_{1})$ depends only
on $\Omega$.
\end{proof}
Finally, in the next two lemmas, we investigate the validity of \hypref{main}
without symmetry assumptions for the two critically decaying harmonic
functions. We prove that \hypref{main} is verified for the flux carrier
under some conditions but never for the pure rotating solution.
\begin{lem}[{\citealp[Lemma 3]{Russo-existenceofDsolutions2011}}]
\label{lem:flux}For $\Phi\in\mathbb{R}$, if
\[
\bar{\bu}_{\Phi}=\frac{\Phi}{2\pi}\frac{\be_{r}}{r}\,,
\]
then for all $\bv\in\dot{H}_{0,\sigma}^{1}(\Omega)$,
\[
\bigl|\big(\bv\bcdot\bnabla\bv,\bar{\bu}_{\Phi}\big)\big|\le\frac{\left|\Phi\right|}{2\pi}\left\Vert \bnabla\bv\right\Vert _{2}^{2}\,.
\]
\end{lem}
\begin{proof}
We have
\[
\bar{\bu}_{\Phi,0,0}=\frac{\Phi}{2\pi}\bnabla\log r\,,
\]
so by integrating by part, we have
\[
\big(\bv\bcdot\bnabla\bv,\bar{\bu}_{\Phi,0,0}\big)=\frac{\Phi}{2\pi}\big(\bnabla\bcdot\left(\bv\bcdot\bnabla\bv\right),\log r\big)\,.
\]
By using \citet[Theorem II.1]{Coifman-Compensatedcompactness1993},
$\bnabla\bcdot\left(\bv\bcdot\bnabla\bv\right)=\bnabla\bv:\left(\bnabla\bv\right)^{T}$
is in the Hardy space $\mathcal{H}^{1}(\Omega)$, and since $\log r\in\BMO(\Omega)$,
we obtain that $\big(\bv\bcdot\bnabla\bv,\bar{\bu}_{\Phi}\big)$ is
a continuous bilinear form on $\dot{H}_{0,\sigma}^{1}(\Omega)$. An
explicit calculation performed by \citet[Lemma 3]{Russo-existenceofDsolutions2011}
and reproduced in \citet[Remark X.4.2]{Galdi-IntroductiontoMathematical2011}
determines the value of the constant,
\[
\big|\big(\bv\bcdot\bnabla\bv,\bar{\bu}_{\Phi,0,0}\big)\big|\leq\frac{\left|\Phi\right|}{2\pi}\left\Vert \bnabla\bv\right\Vert _{2}^{2}\,.
\]

\end{proof}
The following lemma is given as an example that not all critically
decaying steady solutions satisfy \hypref{main}, but can be skipped
as it will not be used later on.
\begin{lem}
For $\mu\neq0$, if
\[
\bar{\bu}_{0,\mu,0}=\frac{\mu}{2\pi}\frac{\be_{\theta}}{r}\,,
\]
then for any $\delta>0$, there exists $\bv\in H_{0,\sigma}^{1}(\Omega)$
such that
\[
\big(\bv\bcdot\bnabla\bv,\bar{\bu}_{0,\mu,0}\big)\geq\delta\left\Vert \bnabla\bv\right\Vert _{2}^{2}\,.
\]
\end{lem}
\begin{proof}
Without loss of generality, we assume $\mu=2\pi$. For $\alpha>0$
small, we define the following stream function and the corresponding
velocity field,
\begin{align*}
\psi_{\alpha} & =r^{\cos\alpha}\cos(\theta-\sin\alpha\,\log r)\,, & \bu_{\alpha} & =\bnabla\bwedge\psi_{\alpha}\,.
\end{align*}
Then an explicit calculation, shows that
\begin{equation}
\int_{\mathbb{R}^{2}\setminus B(\bzero,1)}\left|\bnabla\bu_{\alpha}\right|^{2}=4\pi\,.\label{eq:ualpha-int1}
\end{equation}
On the other hand, we have
\[
\bu_{\alpha}\bcdot\bnabla\bu_{\alpha}\bcdot\bar{\bu}_{0,\mu,0}=\frac{\sin2\alpha}{2}r^{-4+2\cos\alpha}\,,
\]
so
\begin{equation}
\int_{\mathbb{R}^{2}\setminus B(\bzero,1)}\left(\bu_{\alpha}\bcdot\bnabla\bu_{\alpha}\bcdot\bar{\bu}_{0,\mu,0}\right)=\frac{\pi\sin2\alpha}{2-2\cos\alpha}\,.\label{eq:ualpha-int2}
\end{equation}
Since $\bigl\Vert\bu_{\alpha}\bigr\Vert_{\infty}$ and $\bigl\Vert\left|\bx\right|^{-1}\psi_{\alpha}\bigr\Vert_{\infty}$
are uniformly bounded in $\alpha$, if $\bu_{\alpha}$ were in $H_{0,\sigma}^{1}(\Omega)$,
the lemma would be proven by taking $\alpha\to0$. In the following,
we will make a suitable cutoff of $\bu_{\alpha}$ to make its support
compact in $\Omega$.

Let $\chi:\Omega\to[0,1]$ be a smooth cutoff function such that $\chi=0$
on $\Omega\cap\bar{B}(\bzero,R)$ and $\chi=1$ on $\Omega\setminus B(\bzero,2R)$.
Let $\eta:\mathbb{R}\to[0,1]$ be a smooth cutoff function such that
$\eta(r)=1$ if $r\leq1$ and $\eta(r)=0$ if $r\geq2$. For $k>0$
large enough, let $\eta_{k}$ be defined by
\[
\eta_{k}(\bx)=\chi(\bx)\,\eta\left(\frac{\log\log\left|\bx\right|}{\log\log k}\right)\,.
\]
Therefore, $\eta_{k}$ has support in $B(\bzero,K)\setminus\bar{B}(\bzero,R)$
where $K=\e^{\log^{2}k}$ and $\eta_{k}(\bx)=1$ if $2R\leq\left|\bx\right|\leq k$.
Moreover the decay at infinity is the following,
\begin{equation}
\bigl|\bnabla\eta_{k}(\bx)\bigr|+\left|\bx\right|\bigl|\bnabla^{2}\eta_{k}(\bx)\bigr|\leq\frac{C_{\eta}}{\log\log k}\frac{1}{\left|\bx\right|\log\left|\bx\right|}\,,\label{eq:cutoff-eta}
\end{equation}
where $C_{\eta}>0$ is independent of $k$. We now define
\[
\bu_{\alpha,k}=\bnabla\bwedge\left(\eta_{k}\psi_{\alpha}\right)=\eta_{k}\bu_{\alpha}+\psi_{\alpha}\left(\bnabla\bwedge\eta_{k}\right)\,,
\]
so
\[
\bnabla\bu_{\alpha,k}=\eta_{k}\bnabla\bu_{\alpha}+\bnabla\eta_{k}\otimes\bu_{\alpha}+\bu_{\alpha}^{\perp}\otimes\left(\bnabla\bwedge\eta_{k}\right)+\psi_{\alpha}\bnabla\left(\bnabla\bwedge\eta_{k}\right)\,.
\]
Since $\bigl\Vert\bu_{\alpha}\bigr\Vert_{\infty}$ and $\bigl\Vert\left|\bx\right|^{-1}\psi_{\alpha}\bigr\Vert_{\infty}$
are uniformly bounded in $\alpha$, in view of \eqref{cutoff-eta}
and \eqref{ualpha-int1}, we have 
\begin{equation}
\left\Vert \bnabla\bu_{\alpha,k}\right\Vert _{2}\leq\left\Vert \eta_{k}\right\Vert _{\infty}\left\Vert \bnabla\bu_{\alpha}\right\Vert _{2}+\left\Vert \bu_{\alpha}\right\Vert _{\infty}\left\Vert \bnabla\eta_{k}\right\Vert _{2}+\bigl\Vert\left|\bx\right|^{-1}\psi_{\alpha}\bigr\Vert_{\infty}\bigl\Vert\left|\bx\right|\bnabla^{2}\eta_{k}\bigr\Vert_{2}\leq C\,,\label{eq:bound-grad-u-alpha-k}
\end{equation}
for some $C_{\bu}>0$ independent of $\alpha$ and $k$.

To analyze the other term, we split the domain $\Omega$ into $\Omega_{1}=\Omega\cap B(\bzero,2R)$
and $\Omega_{2}=\mathbb{R}^{2}\setminus B(\bzero,2R)$ and define
\[
I_{i}(\alpha,k)=\int_{\Omega_{i}}\big(\bu_{\alpha,k}\bcdot\bnabla\bu_{\alpha,k}\bcdot\bar{\bu}_{0,\mu,0}\big)\,,
\]
for $i=1,2$. The functions are independent of $k$ in $\Omega_{1}$,
so there exists a constant $C_{1}>0$ independent of $\alpha$ and
$k$ such that
\[
\left|I_{1}(\alpha,k)\right|\leq C_{1}\,.
\]
In $\Omega_{2}$, the cutoff-function $\eta_{k}$ is radially symmetric
and an explicit calculation shows that
\begin{align*}
I_{2}(\alpha,k) & =2\pi\int_{2R}^{\infty}r^{-3+2\cos\alpha}\left(\sin2\alpha\,\eta_{k}^{2}+\sin\alpha\,\eta_{k}\, r\bnabla\eta_{k}\bcdot\be_{r}\right)\rd r\\
 & \geq2\pi\sin2\alpha\,\int_{2R}^{k}r^{-3+2\cos\alpha}\rd r-\frac{2\pi C_{\eta}\sin\alpha}{\log\log k}\left|\int_{k}^{K}r^{-3+2\cos\alpha}\frac{\rd r}{\log r}\right|\,,
\end{align*}
where we used \eqref{cutoff-eta} for the last step. We have
\[
\left|\int_{k}^{K}r^{-3+2\cos\alpha}\frac{\rd r}{\log r}\right|\leq\left|\int_{\e}^{K}\frac{\rd r}{r\log r}\right|\leq\log\log K\leq2\log\log k\,,
\]
and moreover, by choosing
\begin{align}
k_{\alpha} & =2^{\frac{1}{2(1-\cos\alpha)}}\,,\label{eq:def-k-alpha}
\end{align}
we have
\[
\int_{1}^{k_{\alpha}}r^{-3+2\cos\alpha}\rd r=\frac{1}{4-4\cos\alpha}\,.
\]
Therefore, the exists $C_{2}>0$ independent of $\alpha$, such that
\[
I_{2}(\alpha,k_{\alpha})\geq\frac{\pi\sin2\alpha}{2-2\cos\alpha}-C_{2}\,.
\]
Defining $\bv_{\alpha}=\bu_{\alpha,k_{\alpha}}$, we have
\[
\big(\bv_{\alpha}\bcdot\bnabla\bv_{\alpha},\bar{\bu}_{0,\mu,0}\big)=I_{1}(\alpha,k_{\alpha})+I_{2}(\alpha,k_{\alpha})\geq\frac{\pi\sin2\alpha}{2-2\cos\alpha}-C_{1}-C_{2}\,,
\]
and therefore the lemma is proven by taking the limit $\alpha\to0^{+}$
since $\left\Vert \bnabla\bv_{\alpha}\right\Vert _{2}\leq C$ by \eqref{bound-grad-u-alpha-k}.
\end{proof}
All the previous lemmas, lead to the following corollaries of \thmref{stability}:
\begin{cor}
Let assume that $\bar{\bu}=\bar{\bu}_{\Phi}+\bar{\bu}_{\delta}\in\dot{H}_{\sigma}^{1}(\Omega)$,
where $\bar{\bu}_{\delta}$ satisfies \eqref{subcritical} for some
$\delta>0$, is a steady solution of \eqref{ns-ubar}. If $\left|\Phi\right|+2\pi\delta<2\pi$,
then $\bar{\bu}$ is asymptotically stable with respect to any $L^{2}$-perturbation.\end{cor}
\begin{proof}
By \lemref{hardy,flux}, \hypref{main} is satisfied, so \thmref{stability}
holds.\end{proof}
\begin{cor}
Let assume that the domain $\Omega$ is invariant under the central
symmetry $\bx\mapsto-\bx$, and that $\bar{\bu}\in\dot{H}_{\sigma}^{1}(\Omega)$
is a centrally symmetric steady solution of \eqref{ns-ubar} satisfying
\eqref{critical} for some $\delta>0$. If $\delta<1$, then $\bar{\bu}$
is asymptotically stable with respect to any centrally symmetric $L^{2}$-perturbation.\end{cor}
\begin{proof}
The Navier-Stokes equations \eqref{ns-u,ns-ubar} are invariant under
the central symmetry and therefore, all the spaces can be restricted
to centrally symmetric solutions. Then the result follows by applying
\lemref{hardy-central} and \thmref{stability}.
\end{proof}

\section{Existence of a weak solution}

The proof of the existence of weak solutions follows the standard
method for showing the existence of weak solutions, see for example
\citet[Theorem 3.1]{Temam-Navier-StokesEquations1977}, so only the
main steps and sketched below.
\begin{proof}[Proof of \thmref{existence}]
Let $\left\{ \bphi_{i}\right\} _{i\geq1}\subset C_{0,\sigma}^{\infty}(\Omega)$
be a sequence which is dense and orthonormal in $H_{0,\sigma}^{1}(\Omega)$.
For each $n$, we define an approximate solution $\bv_{n}$ with initial
data,
\begin{align*}
\bv_{n} & =\sum_{i=1}^{n}\xi_{in}\bphi_{i}\,, & \bv_{n}(0) & =\sum_{i=1}^{n}\left(\bv_{0},\bphi_{i}\right)\bphi_{i}\,,
\end{align*}
which satisfies
\begin{equation}
\left(\partial_{t}\bv_{n},\bphi_{i}\right)+\left(\bnabla\bv_{n},\bnabla\bphi_{i}\right)+\left(\bar{\bu}\bcdot\bnabla\bv_{n},\bphi_{i}\right)+\left(\bv_{n}\bcdot\bnabla\bar{\bu},\bphi_{i}\right)+\left(\bv_{n}\bcdot\bnabla\bv_{n},\bphi_{i}\right)=0\,,\label{eq:approximate-solution}
\end{equation}
for all $i\in\{0,\dots,n\}$. We now obtain the a priori estimates.
By multiplying \eqref{approximate-solution} by $\xi_{in}$ and summing
over $i$,
\[
\left(\partial_{t}\bv_{n},\bv_{n}\right)+\left(\bnabla\bv_{n},\bnabla\bv_{n}\right)+\left(\bv_{n}\bcdot\bnabla\bar{\bu},\bv_{n}\right)=0\,.
\]
Therefore by using \hypref{main}, we have
\begin{equation}
\frac{1}{2}\frac{\rd}{\rd t}\left\Vert \bv_{n}\right\Vert _{2}^{2}+\left(1-\delta\right)\left\Vert \bnabla\bv_{n}\right\Vert _{2}^{2}\leq0\,.\label{eq:energy-inequality-approximate}
\end{equation}
Let $i\geq1$ and $\Omega_{i}$ denotes the support of $\bphi_{i}$.
Therefore the sequence $\left\{ \bv_{n}\right\} _{n\geq1}$ is bounded
in $X_{T}$, so there exists a sub-sequence, also denoted by $\left\{ \bv_{n}\right\} _{n\geq1}$,
which converges to $\bv\in X_{T}$ weakly in $L^{2}(0,T;\dot{H}_{0,\sigma}^{1}(\Omega_{i}))$
and strongly in $L^{2}(0,T;L^{2}(\Omega_{i}))$. Multiplying \eqref{approximate-solution}
by $\psi\in C^{1}([0,T])$, we obtain
\[
-\left(\bv_{n},\dot{\psi}\bphi_{i}\right)+\left(\bnabla\bv_{n},\bnabla\psi\bphi_{i}\right)+\left(\bar{\bu}\bcdot\bnabla\bv_{n},\psi\bphi_{i}\right)+\left(\bv_{n}\bcdot\bnabla\bar{\bu},\psi\bphi_{i}\right)+\left(\bv_{n}\bcdot\bnabla\bv_{n},\psi\bphi_{i}\right)=0\,.
\]
By integrating by parts, the third and the last term, we can pass
to the limit and obtain the existence of a weak solution $\bv$ in
$X_{T}$. By integrating the energy inequality for the approximate
solution \eqref{energy-inequality-approximate} and passing to the
limit, we obtain \eqref{energy-inequality}.
\end{proof}

\section{Linearized system}

In this section, we study the linear system
\begin{equation}
\begin{aligned}\partial_{t}\bv+\bar{\bu}\bcdot\bnabla\bv+\bv\bcdot\bnabla\bar{\bu} & =\Delta\bv-\bnabla q\,, & \bnabla\bcdot\bv & =0\,,\\
\bv\big|_{t=0} & =\bv_{0}\,, & \bv\big|_{\partial\Omega} & =\bzero\,.
\end{aligned}
\label{eq:ns-linear}
\end{equation}
Throughout this section we always assume that \hypref{main} holds.
If $\mathbb{P}$ denotes the Leray projection, we define the following
operator
\[
\mathcal{L}\bv=-\Delta\bv+\mathbb{P}\left(\bar{\bu}\bcdot\bnabla\bv\right)+\mathbb{P}\left(\bv\bcdot\bnabla\bar{\bu}\right)\,,
\]
and its adjoint on $L_{\sigma}^{2}(\Omega)$,
\[
\mathcal{L}^{*}\bv=-\Delta\bv-\bar{\bu}\bcdot\bnabla\bv+(\bnabla\bar{\bu})^{T}\bv\,.
\]
More precisely they are defined by the following bilinear forms on
$H_{0,\sigma}^{1}(\Omega)$,
\begin{align*}
a_{\mathcal{L}}(\bv,\bphi) & =\big(\bnabla\bv,\bnabla\bphi\bigr)+\bigl(\bar{\bu}\bcdot\bnabla\bv,\bphi\bigr)+\bigl(\bv\bcdot\bnabla\bar{\bu},\bphi\bigr)\,,\\
a_{\mathcal{L}^{*}}(\bv,\bphi) & =\big(\bnabla\bv,\bnabla\bphi\bigr)-\bigl(\bar{\bu}\bcdot\bnabla\bv,\bphi\bigr)+\bigl(\bphi\bcdot\bnabla\bar{\bu},\bv\bigr)\,.
\end{align*}
We show that $-\mathcal{L}$ and $-\mathcal{L}^{*}$ generate strongly
continuous semigroups on $L_{\sigma}^{2}(\Omega)$ and some standard
corollaries:
\begin{prop}
The closure in $L_{\sigma}^{2}(\Omega)$ of the operators $\mathcal{L}$
and $\mathcal{L}^{*}$ are infinitesimal generators of analytic semigroups
on $L_{\sigma}^{2}(\Omega)$.\end{prop}
\begin{proof}
First of all, the forms $a_{\mathcal{L}}$ and $a_{\mathcal{L}^{*}}$
are bounded on $H_{0,\sigma}^{1}(\Omega)$,
\begin{align*}
\left|a_{\mathcal{L}}(\bv,\bphi)\right| & \leq\left\Vert \bnabla\bv\right\Vert _{2}\left\Vert \bnabla\bphi\right\Vert _{2}+\left\Vert \bar{\bu}\right\Vert _{\infty}\left\Vert \bnabla\bv\right\Vert _{2}\left\Vert \bphi\right\Vert _{2}+\left\Vert \bnabla\bar{\bu}\right\Vert _{2}\left\Vert \bv\right\Vert _{4}\left\Vert \bphi\right\Vert _{4}\,,\\
\left|a_{\mathcal{L}^{*}}(\bv,\bphi)\right| & \leq\left\Vert \bnabla\bv\right\Vert _{2}\left\Vert \bnabla\bphi\right\Vert _{2}+\left\Vert \bar{\bu}\right\Vert _{\infty}\left\Vert \bnabla\bv\right\Vert _{2}\left\Vert \bphi\right\Vert _{2}+\left\Vert \bnabla\bar{\bu}\right\Vert _{2}\left\Vert \bv\right\Vert _{4}\left\Vert \bphi\right\Vert _{4}\,,
\end{align*}
since $H_{0,\sigma}^{1}(\Omega)$ is continuously embedded in $L^{4}(\Omega)$.
Moreover by using \hypref{main},
\begin{equation}
a_{\mathcal{L}}(\bv,\bv)=a_{\mathcal{L}^{*}}(\bv,\bv)=\left\Vert \bnabla\bv\right\Vert _{2}^{2}-\bigl(\bv\bcdot\bnabla\bv,\bar{\bu}\bigr)\ge\left(1-\delta\right)\left\Vert \bnabla\bv\right\Vert _{2}^{2}\,.\label{eq:accretive}
\end{equation}
Therefore, by using for example \citet[Proposition 4.1]{Karch.Pilarczyk-AsymptoticStabilityof2011}
or the references therein, we obtain that $-\mathcal{L}$ and $-\mathcal{L}^{*}$
generate analytic semigroups on $L_{\sigma}^{2}(\Omega)$.
\end{proof}
We have the following standard corollaries:
\begin{cor}
\label{cor:L}For any $\bv_{0}\in L_{\sigma}^{2}(\Omega)$, we have
\begin{align*}
\bigl\Vert\e^{-t\mathcal{L}}\bv_{0}\bigr\Vert_{2} & \leq\bigl\Vert\bv_{0}\bigr\Vert_{2}\,, & \lim_{t\to\infty}\bigl\Vert\e^{-t\mathcal{L}}\bv_{0}\bigr\Vert_{2} & =0\,, & \lim_{t\to\infty}\frac{1}{t}\int_{0}^{t}\bigl\Vert\e^{-s\mathcal{L}}\bv_{0}\bigr\Vert_{2}\rd s & =0\,.
\end{align*}
\end{cor}
\begin{proof}
The solution $\bv(t)=\e^{-t\mathcal{L}}\bv_{0}$ satisfies the following
energy inequality,
\[
\frac{1}{2}\frac{\rd}{\rd t}\left\Vert \bv\right\Vert _{2}^{2}+\left(1-\delta\right)\left\Vert \bnabla\bv\right\Vert _{2}=0\,.
\]
The first bound is proven by integrating from $0$ to $t$.

For the second bound, since the range of $\mathcal{L}$ is dense in
$L_{\sigma}^{2}(\Omega)$, for every $\varepsilon>0$ there exists
$\bphi\in\mathcal{D}(\mathcal{L})$ such that $\left\Vert \bv_{0}-\mathcal{L}\bphi\right\Vert _{2}\leq\varepsilon$.
Therefore, by using the first bound,
\[
\bigl\Vert\e^{-t\mathcal{L}}\bv_{0}\bigr\Vert_{2}\leq\bigl\Vert\e^{-t\mathcal{L}}\left(\bv_{0}-\mathcal{L}\bphi\right)\bigr\Vert_{2}+\bigl\Vert\mathcal{L}\e^{-t\mathcal{L}}\bphi\bigr\Vert_{2}\leq\varepsilon+Ct^{-1}\left\Vert \bphi\right\Vert \,,
\]
so the second bound is proven.

Substituting $s=\tau t$, we have
\[
\frac{1}{t}\int_{0}^{t}\bigl\Vert\e^{-s\mathcal{L}}\bv_{0}\bigr\Vert_{2}\rd s=\int_{0}^{1}\bigl\Vert\e^{-t\tau\mathcal{L}}\bv_{0}\bigr\Vert_{2}\rd\tau\leq\bigl\Vert\bv_{0}\bigr\Vert_{2}\,,
\]
since the semigroup is contracting. Therefore, the third bound follows
by using the second bound together with the dominated convergence
theorem.
\end{proof}

\begin{cor}
\label{cor:Lstar}For any $\bv_{0}\in L_{\sigma}^{2}(\Omega)$, we
have
\[
\bigl\Vert\bnabla\e^{-t\mathcal{L}^{*}}\bv_{0}\bigr\Vert_{2}\leq Ct^{-1/2}\bigl\Vert\bv_{0}\bigr\Vert_{2}\,.
\]
\end{cor}
\begin{proof}
Using \eqref{accretive}, for $\bv\in L_{\sigma}^{2}(\Omega)$, we
have
\[
\left(1-\delta\right)\left\Vert \bnabla\bv\right\Vert _{2}^{2}\leq a_{\mathcal{L}^{*}}(\bv,\bv)=\bigl(\mathcal{L}^{*}\bv,\bv\bigr)\leq\left\Vert \mathcal{L}^{*}\bv\right\Vert _{2}\left\Vert \bv\right\Vert _{2}\,.
\]
Therefore, taking $\bv=\e^{-t\mathcal{L}^{*}}\bv_{0}$, we obtain
\[
\left\Vert \bnabla\bv\right\Vert _{2}^{2}\leq\left(1-\delta\right)^{-1}\bigl\Vert\mathcal{L}^{*}\e^{-t\mathcal{L}^{*}}\bv_{0}\bigr\Vert_{2}\left\Vert \bv_{0}\right\Vert _{2}\leq C^{2}t^{-1}\left\Vert \bv_{0}\right\Vert _{2}^{2}\,.
\]

\end{proof}

\section{Asymptotic stability}

Using the properties of the linearized system, we now prove the result
on the asymptotic stability (\thmref{stability}).
\begin{lem}
\textup{\label{lem:nonlinear}There exists $C>0$ such that for all
$\bv\in H_{0,\sigma}^{1}(\Omega)$ and $\bphi\in L_{\sigma}^{2}(\Omega)$,
\[
\left(\bv\bcdot\bnabla\bv,\e^{-t\mathcal{L}^{*}}\bphi\right)\leq Ct^{-1/2}\left\Vert \bv\right\Vert _{2}\left\Vert \bnabla\bv\right\Vert _{2}\left\Vert \bphi\right\Vert _{2}\,.
\]
}\end{lem}
\begin{proof}
Using \corref{Lstar}, we have
\[
\left|\left(\bv\bcdot\bnabla\bv,\e^{-t\mathcal{L}^{*}}\bphi\right)\right|=\left|\left(\bv\bcdot\bnabla\e^{-t\mathcal{L}^{*}}\bphi,\bv\right)\right|\leq\left\Vert \bv\right\Vert _{4}^{2}\bigl\Vert\bnabla\e^{-t\mathcal{L}^{*}}\bphi\bigr\Vert_{2}\leq Ct^{-1/2}\left\Vert \bv\right\Vert _{4}^{2}\left\Vert \bphi\right\Vert _{2}\,.
\]
The result follows by applying the Sobolev inequality $\left\Vert \bv\right\Vert _{4}^{4}\leq2\left\Vert \bv\right\Vert _{2}^{2}\left\Vert \bnabla\bv\right\Vert _{2}^{2}$.
\end{proof}
The following lemma will be crucial in the proof of the asymptotic
stability without using any spectral decomposition.
\begin{lem}
\label{lem:limit-zero}For a nonnegative function $f\in L^{2}(0,\infty)$,
we have
\[
\lim_{t\to0}\frac{1}{t}\int_{0}^{t}\int_{0}^{s}(s-\tau)^{-1/2}f(\tau)\,\rd\tau\,\rd s=0\,.
\]
\end{lem}
\begin{proof}
We first prove that
\[
\lim_{t\to\infty}\left\Vert \chi_{t}f\right\Vert _{1}=0\,,\qquad\text{where}\qquad\chi_{t}=\frac{1}{\sqrt{t}}\mathbbm{1}_{[0,t]}\,.
\]
Since $\left\Vert \chi_{t}\right\Vert _{2}=1$ and $\chi_{t}$ converges
pointwise to $0$ as $t\to\infty$, we obtain that $\chi_{t}$ converges
weakly to $0$ in $L^{2}(0,\infty)$ as $t\to\infty$, see \citet[Theorem 13.44]{Hewitt-RealandAbstract1965}.
Therefore $\lim_{t\to\infty}\left\Vert \chi_{t}f\right\Vert _{1}=0$
since $f\in L^{2}(0,\infty)$.

By interchanging the order of integrations, we have
\begin{align*}
I(t) & =\frac{1}{t}\int_{0}^{t}\int_{0}^{s}(s-\tau)^{-1/2}f(\tau)\,\rd\tau\,\rd s=\frac{1}{t}\int_{0}^{t}\int_{\tau}^{t}(s-\tau)^{-1/2}f(\tau)\,\rd s\,\rd\tau\\
 & =\frac{2}{t}\int_{0}^{t}\left(t-\tau\right)^{1/2}f(\tau)\,\rd\tau\leq\frac{2}{\sqrt{t}}\int_{0}^{t}f(\tau)\,\rd\tau=2\left\Vert \chi_{t}f\right\Vert _{1}\,,
\end{align*}
where the Hölder inequality justifies the interchange of the integrations,
since $\left\Vert \chi_{t}f\right\Vert _{1}\leq\left\Vert \chi_{t}\right\Vert _{2}\left\Vert f\right\Vert _{2}\leq\left\Vert f\right\Vert _{2}$.
The lemma is proven since $\lim_{t\to\infty}I(t)\leq2\lim_{t\to\infty}\left\Vert \chi_{t}f\right\Vert _{1}=0$.
\end{proof}

\begin{proof}[Proof of \thmref{stability}]
Let $\bphi\in C_{0,\sigma}^{\infty}(\Omega)$, $s>0$, and
\[
\bpsi(\tau)=\e^{-(s-\tau)\mathcal{L}^{*}}\bphi\,,
\]
for $\tau\in(0,s)$. Since
\begin{align*}
\int_{0}^{s}\bigl(\bv,\dot{\bpsi}\bigr)\rd\tau & =\int_{0}^{s}\bigl(\bv,\mathcal{L}^{*}\bpsi\bigr)\rd\tau=\int_{0}^{s}\bigl(\mathcal{L}\bv,\bpsi\bigr)\rd\tau\\
 & =\int_{0}^{s}\Big[\left(\bnabla\bv,\bnabla\bpsi\right)+\left(\bar{\bu}\bcdot\bnabla\bv,\bpsi\right)+\left(\bv\bcdot\bnabla\bar{\bu},\bpsi\right)\Big]\rd\tau\,,
\end{align*}
by \defref{weak-solution}, we obtain
\[
\big(\bv(s),\bphi\bigr)+\int_{0}^{s}\left(\bv\bcdot\bnabla\bv,\bpsi\right)\rd\tau=\big(\bv_{0},\bpsi(0)\bigr)=\big(\e^{-s\mathcal{L}}\bv_{0},\bphi\bigr)\,.
\]
Taking $\bphi=\bv(s)$ in this expression,
\[
\left\Vert \bv(s)\right\Vert _{2}^{2}=\big(\e^{-s\mathcal{L}}\bv_{0},\bv(s)\bigr)-\int_{0}^{s}\left((\bv\bcdot\bnabla\bv)(\tau),\e^{-(s-\tau)\mathcal{L}^{*}}\bv(s)\right)\rd\tau\,,
\]
so by using \lemref{nonlinear}, we have
\[
\left\Vert \bv(s)\right\Vert _{2}\leq\bigl\Vert\e^{-s\mathcal{L}}\bv_{0}\bigr\Vert_{2}+C\int_{0}^{s}(s-\tau)^{-1/2}\left\Vert \bv(\tau)\right\Vert _{2}\left\Vert \bnabla\bv(\tau)\right\Vert _{2}\rd\tau\,.
\]
By the energy inequality \eqref{energy-inequality}, we have $\left\Vert \bv(\tau)\right\Vert _{2}\leq\left\Vert \bv_{0}\right\Vert _{2}$,
so
\[
\left\Vert \bv(s)\right\Vert _{2}\leq\bigl\Vert\e^{-s\mathcal{L}}\bv_{0}\bigr\Vert_{2}+C\left\Vert \bv_{0}\right\Vert _{2}\int_{0}^{s}(s-\tau)^{-1/2}\left\Vert \bnabla\bv(\tau)\right\Vert _{2}\rd\tau\,.
\]
Integrating on $s$ from $0$ to $t$ and multiplying by $t^{-1}$,
we obtain
\[
\frac{1}{t}\int_{0}^{t}\left\Vert \bv(s)\right\Vert _{2}\rd s\leq\frac{1}{t}\int_{0}^{t}\bigl\Vert\e^{-s\mathcal{L}}\bv_{0}\bigr\Vert_{2}\,\rd s+C\left\Vert \bv_{0}\right\Vert _{2}\int_{0}^{t}\int_{0}^{s}(s-\tau)^{-1/2}\left\Vert \bnabla\bv(\tau)\right\Vert _{2}\rd\tau\,\rd s\,.
\]
By the energy inequality, $\left\Vert \bnabla\bv(\cdot)\right\Vert _{2}\in L^{2}(0,\infty)$,
so in view of \corref{L} and \lemref{limit-zero} we have
\[
\lim_{t\to\infty}\frac{1}{t}\int_{0}^{t}\left\Vert \bv(s)\right\Vert _{2}\rd s=0\,.
\]
Finally, since $\left\Vert \bv(t)\right\Vert _{2}$ is a non-increasing
function of $t\geq0$, we obtain for $t>0$,
\[
\left\Vert \bv(t)\right\Vert _{2}\leq\frac{1}{t}\left\Vert \bv(t)\right\Vert _{2}\int_{0}^{t}\rd s\leq\frac{1}{t}\int_{0}^{t}\left\Vert \bv(s)\right\Vert _{2}\rd s\,,
\]
so the proof is complete.
\end{proof}

\subsubsection*{Acknowledgments}

The author would like to thank J. Feneuil for fruitful discussions
and T. Gallay for valuable comments on a preliminary version of the
manuscript. This research was supported by the Swiss National Science
Foundation grant \href{http://p3.snf.ch/Project-161996}{161996}.

\newpage{}

\bibliographystyle{merlin-dot}
\phantomsection\addcontentsline{toc}{section}{\refname}\bibliography{paper}

\end{document}